\documentclass[11pt,twoside, final]{amsart}
\copyrightinfo{0}{Iranian Mathematical Society}
\usepackage{amsmath,amsthm,amscd,amsfonts,amssymb,enumerate}
\usepackage{graphicx}		
\usepackage{color}
\usepackage[colorlinks]{hyperref}

\newtheorem{theorem}{Theorem}[section]
\newtheorem{proposition}[theorem]{Proposition}
\newtheorem{lemma}[theorem]{Lemma}
\newtheorem{corollary}[theorem]{Corollary}
\theoremstyle{definition}

\newtheorem{example}[theorem]{Example}
\newtheorem{problem}[theorem]{Problem}
\theoremstyle{remark}
\newtheorem{remark}[theorem]{Remark}
\numberwithin{equation}{section}
 
 \commby{Hamid Pezeshk}
\begin{document}

 
\title[$2$-irreducible and strongly $2$-irreducible ideals]{$2$-irreducible and strongly $2$-irreducible ideals of commutative rings} 
 
\author[Mostafanasab]{H. Mostafanasab}
\address[Hojjat Mostafanasab]{Department of Mathematics and Applications, University of Mohaghegh Ardabili, P. O. Box 179, Ardabil, Iran}
\email{h.mostafanasab@uma.ac.ir, h.mostafanasab@gmail.com}

\author[Yousefian Darani]{A. Yousefian Darani$^*$}
\address[Ahmad Yousefian Darani]{Department of Mathematics and Applications, University of Mohaghegh Ardabili, P. O. Box 179, Ardabil, Iran}
\email{yousefian@uma.ac.ir, youseffian@gmail.com}

  \thanks{$^*$Corresponding author}
%
\maketitle
%

\begin{abstract}
An ideal $I$ of a commutative ring $R$ is said to be {\it irreducible} if it cannot be 
written as the intersection of two larger ideals. A proper ideal $I$ of a ring 
$R$ is said to be {\it strongly irreducible} if for each ideals $J,~K$ of $R$, 
$J\cap K\subseteq I$ implies that $J\subseteq I$ or 
$K\subseteq I$. In this paper, we introduce the
concepts of 2-irreducible and strongly 2-irreducible ideals which are generalizations of  
irreducible and strongly irreducible ideals, respectively.
We say that a proper ideal $I$ of a ring $R$ is {\it 2-irreducible} if for each ideals 
$J,~K$ and $L$ of $R$, $I=J\cap K\cap L$ implies that either $I=J\cap K$ or 
$I=J\cap L$ or $I=K\cap L$. A proper ideal $I$ of a ring $R$ is called {\it strongly 2-irreducible}
if for each ideals $J,~K$ and $L$ of $R$, $J\cap K\cap L\subseteq I$ implies that either $J\cap K\subseteq I$ or 
$J\cap L\subseteq I$ or $K\cap L\subseteq I$.\\

\end{abstract}
 
\section{\bf Introduction}

Throughout this paper all rings are commutative with a nonzero identity.
Recall that an ideal $I$ of a commutative ring $R$ is {\it irreducible} if $I=J\cap K$ for ideals $J$ and $K$ of $R$
implies that either $I=J$ or $I=K$. 
A proper ideal $I$ of a ring $R$ is said to be {\it strongly irreducible}
if for each ideals $J,~K$ of $R$, $J\cap K\subseteq I$ implies that $J\subseteq I$ or 
$K\subseteq I$ (see \cite{Aziz}, \cite{hei}). Obviously a proper ideal $I$ of a ring $R$ is
strongly irreducible if and only if for each $x,y\in R$, $Rx\cap Ry\subseteq I$ implies that $x\in I$ or
$y\in I$. It is easy to see that any strongly irreducible ideal is an irreducible ideal.
Now, we recall some definitions which are the motivation of our work. Badawi in \cite{B} generalized the concept of prime ideals in a different
way. He defined a nonzero proper ideal $I$ of $R$ to be a {\it 2-absorbing
ideal} of $R$ if whenever $a, b, c\in R$ and $abc\in I$, then $ab\in I$ or $%
ac\in I$ or $bc\in I$.  It is shown that a proper
ideal $I$ of $R$ is a 2-absorbing ideal if and only if whenever $I_1I_2I_3\subseteq I$ for some ideals
$I_1,I_2,I_3$ of $R$, then $I_1I_2\subseteq I$ or $I_1I_3\subseteq I$ or $I_2I_3\subseteq I$. 
In \cite{YFP}, Yousefian Darani and Puczy{\l}owski studied the concept of 2-absorbing commutative semigroups.
Anderson and Badawi 
\cite{AB1} generalized the concept of $2$-absorbing ideals to $n$-absorbing
ideals. According to their definition, a proper ideal $I$ of $R$ is called
an $n$-{\it absorbing} (resp. {\it strongly $n$-absorbing}) ideal if whenever $%
a_1\cdots a_{n+1}\in I$ for $a_1,...,a_{n+1}\in R$ (resp. $I_1\cdots
I_{n+1}\subseteq I$ for ideals $I_1, \cdots I_{n+1}$ of $R$), then there are 
$n$ of the $a_i$'s (resp. $n$ of the $I_i$'s) whose product is in $I$. Thus a strongly 1-absorbing ideal is just a prime ideal. Clearly a strongly $n$-absorbing ideal of $R$ is also an $n$-absorbing ideal of $R$. The concept of 2-absorbing primary ideals, 
a generalization of primary ideals was introduced and investigated in \cite{Bt}.
A proper ideal $I$ of a commutative ring $R$ is called a {\it 2-absorbing primary ideal} if whenever
$a,b,c\in R$ and $abc\in I$, then either $ab\in I$ or $ac\in\sqrt{I}$ or $bc\in\sqrt{I}$.
We refer the readers to \cite{YB} for a specific kind of 2-absorbing ideals and to \cite{M}, \cite{YF}, \cite{YF2} for the module version of the above definitions.
We define an ideal $I$ of a ring $R$ to be {\it 2-irreducible} if whenever $I=J\cap K\cap L$ for ideals $I,~J$ and $K$ of $R$, then either $I=J\cap K$ or $I=J\cap L$ or $I=K\cap L$. Obviously,
any irreducible ideal is a 2-irreducible ideal. Also, we say that a proper ideal $I$ of a ring $R$ is called {\it strongly 2-irreducible} 
if for each ideals $J,~K$ and $L$ of $R$, $J\cap K\cap L\subseteq I$ implies that 
$J\cap K\subseteq I$ or $J\cap L\subseteq I$ or $K\cap L\subseteq I$. Clearly, any strongly irreducible ideal is a strongly 2-irreducible ideal. 
In \cite{YFM2}, \cite{YFM1} we can find the notion of 2-irreducible preradicals and its dual, the notion of co-2-irreducible preradicals.
We call a proper ideal $I$ of a ring $R$ {\it singly strongly 2-irreducible} if for each $x,y,z\in R$, $Rx\cap Ry\cap Rz\subseteq I$ implies that $Rx\cap Ry\subseteq I$ or
$Rx\cap Rz\subseteq I$ or $Ry\cap Rz\subseteq I$. It is trivial that any strongly 2-irreducible ideal is a singly strongly 2-irreducible ideal.
A ring $R$ is said to be an {\it arithmetical ring}, if for each ideals $I,~J$ and $K$ of $R$, $(I+J)\cap K=(I\cap K)+(J\cap K)$. This condition is equivalent to the condition that for each ideals $I,~J$ and $K$ of $R$, $(I\cap J)+K=(I+K)\cap(J+K)$, see \cite{jen}.
In this paper we prove that, a nonzero ideal $I$ of a principal ideal domain $R$ is 2-irreducible if and only if
$I$ is strongly 2-irreducible if and only if $I$ is 2-absorbing primary. It is shown that a proper ideal $I$
of a ring $R$ is strongly 2-irreducible if and only if for each $x,y,z\in R$, $(Rx+Ry)\cap(Rx+Rz)\cap(Ry+Rz)\subseteq I$
implies that $(Rx+Ry)\cap(Rx+Rz)\subseteq I$ or $(Rx+Ry)\cap(Ry+Rz)\subseteq I$
or $(Rx+Rz)\cap(Ry+Rz)\subseteq I$. A proper ideal $I$ of a von Neumann regular ring $R$ is 2-irreducible if and only if $I$ is 2-absorbing
if and only if for every idempotent elements $e_1,e_2,e_3$ of $R$, 
$e_1e_2e_3\in I$ implies that either $e_1 e_2\in I$
or $e_1 e_3\in I$ or $e_2 e_3\in I$. If $I$ is a 2-irreducible ideal of a Noetherian ring $R$, 
then $I$ is a 2-absorbing primary ideal of $R$. Let $R=R_1\times R_2$, where $R_1$ and $R_2$ are commutative rings with
$1\neq0$. It is shown that  a proper ideal $J$ of $R$ is a strongly 2-irreducible ideal of $R$ if and only if
either $J=I_1\times R_2$ for some strongly 2-irreducible ideal $I_1$ of $R_1$ or $J=R_1\times I_2$
for some strongly 2-irreducible ideal $I_2$ of $R_2$ or $J=I_1\times I_2$ for
some strongly irreducible ideal $I_1$ of $R_1$ and some strongly irreducible ideal $I_2$ of $R_2$.
A proper ideal $I$ of a unique factorization domain $R$ is singly strongly 2-irreducible if and only if $p_1^{n_1}p_2^{n_2}\cdots p_k^{n_k}\in I$, where $p_i$'s are distinct prime elements of $R$ and $n_i$'s are natural numbers, implies that $p_r^{n_r}p_s^{n_s}\in I$,
for some $1\leq r,s\leq k$.
\section{\bf Basic properties of 2-irreducible and strongly 2-irreducible ideals}

It is important to notice that when $R$ is a domain, then $R$ is an arithmetical ring if and only if $R$ is a Pr\"{u}fer domain. In particular, every Dedekind domain is an arithmetical domain.
\begin{theorem}\label{basic5}
Let $R$ be a Dedekind domain and $I$ be a nonzero proper ideal of $R$. The following conditions are equivalent:
\begin{enumerate}
\item $I$ is a strongly irreducible ideal;
\item $I$ is an irreducible ideal;
\item $I$ is a primary ideal;
\item $I=Rp^n$ for some prime (irreducible) element $p$ of $R$ and some natural number $n$.
\end{enumerate}
\end{theorem}
\begin{proof}
See  {\rm\cite[Lemma 2.2(3)]{hei}} and {\rm\cite[p. 130, Exercise 36]{L}}.
\end{proof}

We recall from \cite{AA} that an integral domain $R$ is called a $GCD$-domain if any two nonzero elements of $R$ 
have a greatest common divisor $(GCD)$, equivalently, any two nonzero elements of $R$ have a least common multiple $(LCM).$
Unique factorization domains ($UFD$'s) are well-known examples of $GCD$-domains.
Let $R$ be a $GCD$-domain. The least common multiple of elements $x,~y$ of $R$ is denoted by $[x,y]$.
Notice that for every elements $x,~y\in R$, $Rx\cap Ry=R[x,y]$.
Moreover, for every elements $x,y,z$ of $R$, we have
$[[x,y],z]=[x,[y,z]]$. So we denote $[[x,y],z]$ simply by $[x,y,z]$.\\

Recall that every principal ideal domain ($PID$) is a Dedekind domain.
\begin{theorem}\label{basic8}
Let $R$ be a $PID$ and $I$ be a nonzero proper ideal of $R$. The following conditions are equivalent:
\begin{enumerate}
\item $I$ is a 2-irreducible ideal;
\item $I$ is a 2-absorbing primary ideal;
\item Either $I=Rp^k$ for some prime (irreducible) element $p$ of $R$ and some natural number $n$,
or $I=R(p_1^np_2^m)$ for some distinct prime (irreducible) elements $p_1,~p_2$ of $R$ and some natural 
numbers $n,~m$.
\end{enumerate}
\end{theorem}
\begin{proof}
(2)$\Leftrightarrow$(3) See {\rm\cite[Corollary 2.12]{Bt}}.\\
(1)$\Rightarrow$(3) Assume that $I=Ra$ where $0\neq a\in R$. Let 
$a=p_1^{n_1}p_2^{n_2}\cdots p_k^{n_k}$ be a prime decomposition for $a$. We show that either $k=1$ or $k=2$. 
Suppose that $k>2$. By {\rm\cite[p. 141, Exercise 5]{H}}, we have that 
$I=Rp_1^{n_1}\cap Rp_2^{n_2}\cap\cdots\cap Rp_k^{n_k}$. Now, since $I$ is 2-irreducible, there exist $1\leq i,j\leq k$ such that 
$I=Rp_i^{n_i}\cap Rp_j^{n_j}$, say $i=1,~j=2$. Therefore we have 
$I=Rp_1^{n_1}\cap Rp_2^{n_2}\subseteq Rp_3^{n_3}$, which is a contradiction.\\
(3)$\Rightarrow$(1) If $I=Rp^k$ for some prime element $p$ of $R$ and some natural number $n$,
then $I$ is irreducible, by Theorem \ref{basic5}, and so $I$ is 2-irreducible. Therefore, assume that
$I=R(p_1^np_2^m)$ for some distinct prime elements $p_1,~p_2$ of $R$ and some natural 
numbers $n,~m$. Let $I=Ra\cap Rb\cap Rc$ for some elements $a,~b$ and $c$ of $R$. Then
$a,~b$ and $c$ divide $p_1^np_2^m$, and so $a=p_1^{\alpha_1}p_2^{\alpha_2}$, 
$b=p_1^{\beta_1}p_2^{\beta_2}$ and $c=p_1^{\gamma_1}p_2^{\gamma_2}$ 
where  $\alpha_i,\beta_i,\gamma_i$ are some nonnegative integers.
On the other hand $I=Ra\cap Rb\cap Rc=R[a,b,c]=R(p_1^{\delta}p_2^{\varepsilon})$
in which $\delta=max\{\alpha_1,\beta_1,\gamma_1\}$ and $\varepsilon=max\{\alpha_2,\beta_2,\gamma_2\}$.
We can assume without loss of generality that $\delta=\alpha_1$ and $\varepsilon=\beta_2$. So 
$I=R(p_1^{\alpha_1}p_2^{\beta_2})=Ra\cap Rb$. Consequently, $I$ is 2-irreducible.
\end{proof}

A commutative ring $R$ is called a {\it von Neumann regular ring (or an absolutely flat ring)} if for any $a\in R$ there exists an $x\in R$ with $a^2x=a$,  
equivalently, $I=I^2$ for every ideal $I$ of $R$. 
\begin{remark}\label{fully}
Notice that a commutative ring $R$ is a von Neumann regular ring if and only if $IJ=I\cap J$ for any ideals $I,~J$ of $R$, by
{\rm\cite[Lemma 1.2]{jeo}}. Therefore over a commutative von Neumann regular ring the two concepts
of strongly 2-irreducible ideals and of 2-absorbing ideals are coincide.
\end{remark}

\begin{theorem}\label{basic2}
Let $I$ be a proper ideal of a ring $R$. Then the following conditions are equivalent:
\begin{enumerate}
\item $I$ is strongly 2-irreducible;
\item For every elements $x,y,z$ of $R$, $(Rx+Ry)\cap(Rx+Rz)\cap(Ry+Rz)\subseteq I$
implies that $(Rx+Ry)\cap(Rx+Rz)\subseteq I$ or $(Rx+Ry)\cap(Ry+Rz)\subseteq I$
or $(Rx+Rz)\cap(Ry+Rz)\subseteq I$.
\end{enumerate}
\end{theorem}
\begin{proof}
(1)$\Rightarrow$(2) There is nothing to prove.\\
(2)$\Rightarrow$(1) Suppose that $J,~K$ and $L$ are ideals of $R$ such that
neither $J\cap K\subseteq I$ nor $J\cap L\subseteq I$ nor $K\cap L\subseteq I$. Then there exist elements
$x,~y$ and $z$ of $R$ such that $x\in(J\cap K)\backslash I$ and $y\in(J\cap L)\backslash I$ and $z\in(K\cap L)\backslash I$.
On the other hand $(Rx+Ry)\cap(Rx+Rz)\cap(Ry+Rz)\subseteq(Rx+Ry)\subseteq J$,  
$(Rx+Ry)\cap(Rx+Rz)\cap(Ry+Rz)\subseteq(Rx+Rz)\subseteq K$ and 
$(Rx+Ry)\cap(Rx+Rz)\cap(Ry+Rz)\subseteq(Ry+Rz)\subseteq L$. Hence $(Rx+Ry)\cap(Rx+Rz)\cap(Ry+Rz)\subseteq I$,
and so by hypothesis either  $(Rx+Ry)\cap(Rx+Rz)\subseteq I$ or $(Rx+Ry)\cap(Ry+Rz)\subseteq I$
or $(Rx+Rz)\cap(Ry+Rz)\subseteq I$. Therefore, either $x\in I$ or $y\in I$
or $z\in I$, which any of these cases has a contradiction. Consequently $I$ is strongly 2-irreducible.
\end{proof}

A ring $R$ is called a {\it B\'{e}zout ring} if every finitely generated ideal of $R$ is principal. \\
As an immediate consequence of Theorem \ref{basic2} we have the next result:
\begin{corollary}\label{bz1}
Let $I$ be a proper ideal of a B\'{e}zout ring $R$. Then the following conditions are equivalent:
\begin{enumerate}
\item $I$ is strongly 2-irreducible;
\item $I$ is singly strongly 2-irreducible;
\end{enumerate}
\end{corollary}

Now we can state the following open problem.
\begin{problem}
Let $I$ be a singly strongly 2-irreducible ideal of a ring $R$. Is $I$ a strongly 2-irreducible ideal of $R$?
\end{problem}

\begin{proposition}\label{basic6}
Let $R$ be a ring. If $I$ is a strongly 2-irreducible ideal of $R$, then $I$ is a 2-irreducible ideal of $R$.
\end{proposition}
\begin{proof}
Suppose that $I$ is a strongly 2-irreducible ideal of $R$. Let $J,~K$ and $L$ be ideals of $R$ such that
$I=J\cap K\cap L$. Since $J\cap K\cap L\subseteq I$, then either $J\cap K\subseteq I$ or 
$J\cap L\subseteq I$ or $K\cap L\subseteq I$. On the other hand $I\subseteq J\cap K$ and
$I\subseteq J\cap L$ and $I\subseteq K\cap L$. Consequently, either $I= J\cap K$ or
$I= J\cap L$ or $I= K\cap L$. Therefore $I$ is 2-irreducible.
\end{proof}

\begin{remark}
It is easy to check that the zero ideal $I=\{0\}$ of a ring $R$ is 2-irreducible if and only if $I$ is strongly 2-irreducible.
\end{remark}

\begin{proposition}\label{basic7}
Let $I$ be a proper ideal of an arithmetical ring $R$. The following conditions are equivalent:
\begin{enumerate}
\item $I$ is a 2-irreducible ideal of $R$;
\item $I$ is a strongly 2-irreducible ideal of $R$;
\item For every ideals $I_1~,I_2$ and $I_3$ of $R$ with $I\subseteq I_1$, 
$I_1\cap I_2\cap I_3\subseteq I$ implies that $I_1\cap I_2\subseteq I$
or $I_1\cap I_3\subseteq I$ or $I_2\cap I_3\subseteq I$.
\end{enumerate}
\end{proposition}
\begin{proof}
(1)$\Rightarrow$(2) Assume that $J,~K$ and $L$ are ideals of $R$ such that $J\cap K\cap L\subseteq I$.
Therefore $I=I+(J\cap K\cap L)=(I+J)\cap(I+K)\cap(I+L)$, since $R$ is an arithmetical ring. So 
either $I=(I+J)\cap(I+K)$ or $I=(I+J)\cap(I+L)$
or $I=(I+K)\cap(I+L)$, and thus either $J\cap K\subseteq I$ or $J\cap L\subseteq I$
or $K\cap L\subseteq I$. Hence $I$ is a strongly 2-irreducible ideal.\\
(2)$\Rightarrow$(3) is clear.\\
(3)$\Rightarrow$(2) Let $J,~K$ and $L$ be ideals of $R$ such that $J\cap K\cap L\subseteq I$.
Set $I_1:=J+I$, $I_2:=K$ and $I_3:=L$. Since $R$ is an arithmetical ring, then 
$I_1\cap I_2\cap I_3=(J+I)\cap K\cap L=(J\cap K\cap L)+(I\cap K\cap L)\subseteq I$. Hence
either $I_1\cap I_2\subseteq I$ or $I_1\cap I_3\subseteq I$ or $I_2\cap I_3\subseteq I$
which imply that either $J\cap K\subseteq I$ or $J\cap L\subseteq I$ or $K\cap L\subseteq I$,
respectively. Consequently, $I$ is a strongly 2-irreducible ideal of $R$.\\
(2)$\Rightarrow$(1) By Proposition \ref{basic6}.
\end{proof}

As an immediate consequence of  Theorem \ref{basic8} and Proposition \ref{basic7} we have the next result.
\begin{corollary}\label{basic9}
Let $R$ be a $PID$ and $I$ be a nonzero proper ideal of $R$. The following conditions are equivalent:
\begin{enumerate}
\item $I$ is a strongly 2-irreducible ideal;
\item $I$ is a 2-irreducible ideal;
\item $I$ is a 2-absorbing primary ideal;
\item Either $I=Rp^k$ for some prime (irreducible) element $p$ of $R$ and some natural number $n$,
or $I=R(p_1^np_2^m)$ for some distinct prime (irreducible) elements $p_1,~p_2$ of $R$ and some natural 
numbers $n,~m$.
\end{enumerate}
\end{corollary}

The following example shows that the concepts of strongly irreducible (irreducible) ideals and of
strongly 2-irreducible (2-irreducible) ideals are different in general.
\begin{example}
Consider the ideal $6\mathbb{Z}$ of the ring $\mathbb{Z}$. 
By Corollary \ref{basic9}, $6\mathbb{Z}=(2.3)\mathbb{Z}$ is a strongly 2-irreducible
(a 2-irreducible) ideal of $\mathbb{Z}$. But, Theorem \ref{basic5} says that $6\mathbb{Z}$ 
is not a strongly irreducible (an irreducible) ideal of $\mathbb{Z}$.
\end{example}

It is well known that every von Neumann regular ring is a B\'{e}zout ring. By {\rm\cite[p. 119]{jen}}, every
B\'{e}zout ring is an arithmetical ring.

\begin{corollary}
Let $I$ be a proper ideal of a von Neumann regular ring $R$. The following conditions are equivalent:
\begin{enumerate}
\item $I$ is a 2-absorbing ideal of $R$;
\item $I$ is a 2-irreducible ideal of $R$;
\item $I$ is a strongly 2-irreducible ideal of $R$;
\item $I$ is a singly strongly 2-irreducible of $R$;
\item For every idempotent elements $e_1,e_2,e_3$ of $R$, 
$e_1e_2e_3\in I$ implies that either $e_1 e_2\in I$
or $e_1 e_3\in I$ or $e_2 e_3\in I$.
\end{enumerate}
\end{corollary}
\begin{proof}
(1)$\Leftrightarrow$(3) By Remark \ref{fully}.\\
(2)$\Leftrightarrow$(3) By Proposition \ref{basic7}.\\
(3)$\Leftrightarrow$(4) By Corollary \ref{bz1}.\\
(1)$\Rightarrow$(5) is evident.\\
(5)$\Rightarrow$(3) The proof follows from Theorem \ref{basic2} and the fact that
any finitely generated ideal of a von Neumann regular ring $R$ is generated by an idempotent element.
\end{proof}

\begin{proposition}\label{basic}
Let $I_1,~I_2$ be strongly irreducible ideals of a ring $R$. Then $I_1\cap I_2$ is a 
strongly 2-irreducible ideal of $R$.
\end{proposition}
\begin{proof}
Strightforward.
\end{proof}

\begin{theorem}\label{basic11}
Let $R$ be a Noetherian ring. If $I$ is a 2-irreducible ideal of $R$,
then either $I$ is irreducible or $I$ is the intersection of exactly two
irreducible ideals. The converse is true when $R$ is also arithmetical.
\end{theorem}
\begin{proof}
Assume that $I$ is 2-irreducible. By {\rm\cite[Proposition 4.33]{Sh}}, $I$ can be written as a finite irredundant irreducible decomposition $I=I_1\cap I_2\cap\cdots\cap I_k$. We show that either $k=1$ or $k=2$. If $k>3$, then since $I$ is 
2-irreducible, $I=I_i\cap I_j$ for some $1\leq i,j\leq k$, say $i=1$ and $j=2$.
Therefore $I_1\cap I_2\subseteq I_3$, which is a contradiction. For the second atatement, let $R$ be arithmetical, and $I$ be the intersection of two irreducible ideals. Since $R$ is arithmetical, every irreducible ideal is strongly irreducible, {\rm\cite[Lemma 2.2(3)]{hei}}. Now, apply Proposition \ref{basic} to see that $I$ is strongly 2-irreducible, and so $I$ is 2-irreducible. 
\end{proof}

\begin{corollary}
Let $R$ be a Noetherian ring and $I$ be a proper ideal of $R$.
If $I$ is 2-irreducible, then $I$ is a 2-absorbing primary ideal of $R$.
\end{corollary}
\begin{proof}
Assume that $I$ is 2-irreducible. By the fact that every irreducible ideal of a Noetherian ring
is primary and regarding Theorem \ref{basic11}, we have either $I$ is a primary ideal or is the intersection
of two primary ideals. It is clear that every primary ideal is 2-absorbing primary, also the intersection
of two primary ideals is a 2-absorbing primary ideal, by {\rm\cite[Theorem 2.4]{Bt}}.
\end{proof}

\begin{proposition}
Let $R$ be a ring, and let $P_1,~P_2$ and $P_3$ be pairwise comaximal prime ideals of $R$. Then
$P_1P_2P_3$ is not a 2-irreducible ideal.
\end{proposition}
\begin{proof}
The proof is easy.
\end{proof}

\begin{corollary}
If $R$ is a ring such that every proper ideal of $R$ is 2-irreducible, then 
$R$ has at most two maximal ideals.
\end{corollary}

\begin{theorem}
Let $I$ be a radical ideal of a ring $R$, i.e., $I=\sqrt{I}$. The following conditions are equivalent:
\begin{enumerate}
\item $I$ is strongly 2-irreducible;
\item $I$ is 2-absorbing;
\item $I$ is 2-absorbing primary;
\item $I$ is either a prime ideal of $R$ or is an intersection of exactly two prime ideals of $R$.
\end{enumerate}
\end{theorem}
\begin{proof}
(1)$\Rightarrow$(2) Assume that $I$ is strongly 2-irreducible. Let $J,~K$ and $L$ be
ideals of $R$ such that $JKL\subseteq I$. Then $J\cap K\cap L\subseteq\sqrt{J\cap K\cap L}\subseteq\sqrt{I}=I$.
So, either $J\cap K\subseteq I$ or $J\cap L\subseteq I$ or $K\cap L\subseteq I$. Hence either
$JK\subseteq I$ or $JL\subseteq I$ or $KL\subseteq I$. Consequently $I$ is 2-absorbing.\\
(2)$\Leftrightarrow$(3) is obvious.\\
(2)$\Rightarrow$(4) If $I$ is a 2-absorbing ideal, then either $\sqrt{I}$ is a prime ideal or is an intersection of exactly 
two prime ideals, {\rm\cite[Theorem 2.4]{B}}. Now, we prove the claim by assumption that $I=\sqrt{I}$.\\
(4)$\Rightarrow$(1) By Proposition \ref{basic}.
\end{proof}

\begin{theorem}
Let $f:R\rightarrow S$ be a surjective homomorphism of commutative rings, and let 
$I$ be an ideal of $R$ containing $Ker(f)$. Then,
\begin{enumerate}
\item If $I$ is a strongly 2-irreducible ideal of $R$, then $I^e$ is a strongly 2-irreducible ideal of $S$.
\item $I$ is a 2-irreducible ideal of $R$ if and only if $I^e$ is a 2-irreducible ideal of $S$.
\end{enumerate}
\end{theorem}
\begin{proof}
Since $f$ is surjective, $J^{ce}=J$ for every ideal $J$ of $S$. Moreover,
$(K\cap L)^e=K^e\cap L^e$ and $K^{ec}=K$ for every ideals $K,~L$ of $R$ which contain ${\rm Ker(f)}$. \\
(1) Suppose that $I$ is a strongly 2-irreducible ideal of $R$. If $I^e=S$, then $I=I^{ec}=R$, which is a contradiction.
Let $J_1,~J_2$ and $J_3$ be ideals of $S$ such that $J_1\cap J_2\cap J_3\subseteq I^e$.
Therefore $J_1^c\cap J_2^c\cap J_3^c\subseteq I^{ec}=I$. So, either $J_1^c\cap J_2^c\subseteq I$
or $J_1^c\cap J_3^c\subseteq I$ or $J_2^c\cap J_3^c\subseteq I$. Without loss of generality, we may assume
that $J_1^c\cap J_2^c\subseteq I$. So, $J_1\cap J_2=(J_1\cap J_2)^{ce}\subseteq I^e$. Hence $I^e$ is strongly 2-irreducible.\\
(2) The necessity is similar to part (1). Conversely, let $I^e$ be a strongly 2-irreducible ideal of $S$, and let $I_1,~I_2$ and $I_3$
be ideals of $R$ such that $I=I_1\cap I_2\cap I_3$. Then $I^e=I_1^e\cap I_2^e\cap I_3^e$.
Hence, either $I^e=I_1^e\cap I_2^e$ or $I^e=I_1^e\cap I_3^e$
or  $I^e=I_2^e\cap I_3^e$. We may assume that $I^e=I_1^e\cap I_2^e$. Therefore, $I=I^{ec}=I_1^{ec}\cap I_2^{ec}=I_1\cap I_2$. Consequently, $I$ is strongly 2-irreducible.
\end{proof}

\begin{corollary}
Let $f:R\rightarrow S$ be a surjective homomorphism of commutative rings.
There is a one-to-one correspondence between the 2-irreducible ideals of $R$ 
which contain $Ker(f)$ and 2-irreducible ideals of $S$.
\end{corollary}

Recall that a ring $R$ is called a {\it Laskerian ring} if every proper ideal of $R$ has a primary decomposition.
Noetherian rings are some examples of Laskerian rings.

Let $S$ be a multiplicatively closed subset of a ring $R$. In the next theorem, 
consider the natural homomorphism $f:R\rightarrow S^{-1}R$ defined by
$f(x)=x/1$.
\begin{theorem}\label{basic4}
Let $I$ be a proper ideal of a ring $R$ and $S$ be a multiplicatively closed set in $R$.
\begin{enumerate}
\item If $I$ is a strongly 2-irreducible ideal of $S^{-1}R$, 
then $I^c$ is a strongly 2-irreducible ideal of $R$.
\item If $I$ is a  primary strongly 2-irreducible ideal of $R$ such that
$I\cap S=\emptyset$, then $I^e$ is a strongly 2-irreducible ideal of $S^{-1}R$.
\item If $I$ is a primary ideal of $R$ such that $I^{e}$ is a strongly 2-irreducible ideal of $S^{-1}R$,
then $I$ is a strongly 2-irreducible ideal of $R$.
\item If $R^\prime$ is a faithfully flat extension ring of $R$ and if $IR^\prime$
is a strongly 2-irreducible ideal of $R^{\prime}$, then $I$ is a strongly 2-irreducible ideal of $R$.
\item If $I$ is strongly 2-irreducible and $H$ is an ideal of $R$ such that $H\subseteq I$, then $I/H$ is a 
strongly 2-irreducible ideal of $R/H$.
\item If $R$ is a Laskerian ring, then every strongly 2-irreducible ideal is either a primary ideal or is the intersection of two primary ideals.
\end{enumerate}
\end{theorem}
\begin{proof}
(1) Assume that $I$ is a strongly 2-irreducible ideal of $S^{-1}R$.
Let $J,~K$ and $L$ be ideals of $R$ such that $J\cap K\cap L\subseteq I^c$. Then
$J^e\cap K^e\cap L^e\subseteq I^{ce}=I$. Hence either $J^e\cap K^e\subseteq I$
or $J^e\cap L^e\subseteq I$ or $K^e\cap L^e\subseteq I$ since $I$ is 
strongly 2-irreducible. Therefore either $J\cap K\subseteq I^c$
or $J\cap L\subseteq I^c$ or $K\cap L\subseteq I^c$. Consequently 
 $I^c$ is a strongly 2-irreducible ideal of $R$.\\
(2) Suppose that $I$ is a  primary strongly 2-irreducible ideal such that
$I\cap S=\emptyset$. Let $J,~K$ and $L$ be ideals of $S^{-1}R$ such that $J\cap K\cap L\subseteq I^e$.
Since $I$ is a primary ideal, then $J^c\cap K^c\cap L^c\subseteq I^{ec}=I$.
Thus $J^c\cap K^c\subseteq I$ or $J^c\cap L^c\subseteq I$
or $K^c\cap L^c\subseteq I$. Hence $J\cap K\subseteq I^e$ or $J\cap L\subseteq I^e$
or $K\cap L\subseteq I^e$.\\
(3) Let $I$ be a primary ideal of $R$, and let $I^{e}$ be a strongly 2-irreducible ideal of $S^{-1}R$. By part (1), $I^{ec}$ is strongly 2-irreducible. Since $I$ is primary, we have $I^{ec}=I$,
and thus we are done.\\
(4)  Let $J,~K$ and $L$ be ideals of $R$ such that $J\cap K\cap L\subseteq I$. Thus
$JR^\prime\cap KR^\prime\cap LR^\prime=(J\cap K\cap L)R^\prime\subseteq IR^\prime$, by {\rm\cite[Lemma 9.9]{f}}. Since $IR^\prime$ is strongly 2-irreducible, then either $JR^\prime\cap KR^\prime\subseteq IR^\prime$
or $JR^\prime\cap LR^\prime\subseteq IR^\prime$
or $KR^\prime\cap LR^\prime\subseteq IR^\prime$. Without loss of generality, assume that 
$JR^\prime\cap KR^\prime\subseteq IR^\prime$. So,
$(JR^{\prime}\cap R)\cap(KR^\prime\cap R)\subseteq IR^\prime\cap R$. Hence
$J\cap K\subseteq I$, by {\rm\cite[Theorem 4.74]{lam}}. Consequently
$I$ is strongly 2-irreducible.\\
(5) Let $J,~K$ and $L$ be ideals of $R$ containing $H$ such that $(J/H)\cap(K/H)\cap(L/H)\subseteq I/H$.
Hence $J\cap K\cap L\subseteq I$. Therefore, either $J\cap K\subseteq I$ or $J\cap L\subseteq I$ 
or $K\cap L\subseteq I$. Thus, $(J/H)\cap(K/H)\subseteq I/H$ or $(J/H)\cap(L/H)\subseteq I/H$
or $(K/H)\cap(L/H)\subseteq I/H$. Consequently, $I/H$ is strongly 2-irreducible.\\
(6) Let $I$ be a strongly 2-irreducible ideal and $\cap_{i=1}^{n} Q_i$ be a primary decomposition
of $I$. Since $\cap_{i=1}^{n} Q_i\subseteq I$, then there are $1\leq r,s\leq n$ such that 
$Q_r\cap Q_s\subseteq I=\cap_{i=1}^{n} Q_i\subseteq Q_r\cap Q_s$.
\end{proof}

Let $S$ be a multiplicatively closed subset of a ring $R$. Set $$C:=\{I^c\mid I ~\mbox{is an ideal of} ~R_S\}.$$
\begin{corollary}\label{basic10}
Let $R$ be a ring and $S$ be a multiplicatively closed subset of $R$. Then there is a one-to-one 
correspondence between the strongly 2-irreducible ideals of $R_S$ and strongly 2-irreducible ideals of $R$
contained in $C$ which do not meet $S$.
\end{corollary}
\begin{proof}
If $I$ is a strongly 2-irreducible ideal of $R_S$, then evidently $I^c\neq R$, $I^c\in C$ and by
Theorem \ref{basic4}(1), $I^c$ is a strongly 2-irreducible ideal of $R$.
Conversely, let $I$ be a strongly 2-irreducible ideal of $R$, $I\cap S=\emptyset$ and $I\in C$. Since
$I\cap S=\emptyset$, $I^e\neq R_S$. Let $J\cap K\cap L\subseteq I^e$ where $J,~K$ and $L$
are ideals of $R_S$. Then $J^c\cap K^c\cap L^c=(J\cap K\cap L)^c\subseteq I^{ec}$. Now since $I\in C$,
then $I^{ec}=I$. So $J^c\cap K^c\cap L^c\subseteq I$. Hence, either $J^c\cap K^c\subseteq I$
or $J^c\cap L^c\subseteq I$ or $K^c\cap L^c\subseteq I$. Then, either $J\cap K=(J\cap K)^{ce}\subseteq I^e$
or $J\cap L=(J\cap L)^{ce}\subseteq I^e$ or $K\cap L=(K\cap L)^{ce}\subseteq I^e$. Consequently, 
$I^e$ is a strongly 2-irreducible ideal of $R_S$.
\end{proof}

Let $n$ be a natural number. We say that $I$ is an $n$-{\it primary ideal} of a ring $R$ if $I$ is
 the intersection of $n$ primary ideals of $R$.
\begin{proposition}
Let $R$ be a ring. Then the following conditions are equivalent:
\begin{enumerate}
\item Every n-primary ideal of $R$ is a strongly 2-irreducible ideal;
\item For any prime ideal $P$ of $R$, every n-primary ideal of $R_P$ is a strongly 2-irreducible
ideal;
\item For any maximal ideal $m$ of $R$, every n-primary ideal of $R_m$ is a strongly 2-irreducible
ideal.
\end{enumerate}
\end{proposition}
\begin{proof}
(1)$\Rightarrow$(2) Let $I$ be an $n$-primary ideal of $R_P$. We know that $I^c$
is an $n$-primary ideal of $R$, $I^c\cap(R\backslash P)=\emptyset$, $I^c\in C$ and, by the assumption, $I^c$
is a strongly 2-irreducible ideal of $R$. Now, by Corollary \ref{basic10}, $I=(I^c)_P$ is a strongly 2-irreducible ideal of $R_P$.\\
(2)$\Rightarrow$(3) is clear.\\
(3)$\Rightarrow$(1) Let $I$ be an $n$-primary ideal of $R$ and let $m$ be a maximal ideal of $R$
containing $I$ . Then, $I_m$ is an $n$-primary ideal of $R_m$ and so, by our assumption, $I_m$ is a
strongly 2-irreducible ideal of $R_m$. Now by Theorem \ref{basic3}(1), $(I_m)^c$ is a strongly 2-irreducible
ideal of $R$, and since $I$ is an $n$-primary ideal of $R$, $(I_m)^c=I$, that is, $I$ is a strongly 2-irreducible ideal of $R$.
\end{proof}

\begin{theorem}\label{prod}
Let $R=R_1\times R_2$, where $R_1$ and $R_2$ are rings with
$1\neq0$. Let $J$ be a proper ideal of $R$. Then the following conditions are equivalent:
\begin{enumerate}
\item $J$ is a strongly 2-irreducible ideal of $R$;
\item Either $J=I_1\times R_2$ for some strongly 2-irreducible ideal $I_1$ of $R_1$ or $J=R_1\times I_2$
for some strongly 2-irreducible ideal $I_2$ of $R_2$ or $J=I_1\times I_2$ for
some strongly irreducible ideal $I_1$ of $R_1$ and some strongly irreducible ideal $I_2$ of $R_2$.
\end{enumerate}
\end{theorem}
\begin{proof}
(1)$\Rightarrow$(2) Assume that $J$ is a strongly 2-irreducible ideal of $R$. Then
$J=I_1\times I_2$ for some ideal $I_1$ of $R_1$ and some ideal $I_2$ of $R_2$. Suppose that $I_2=R_2$.
Since $J$ is a proper ideal of $R$, $I_1\neq R_1$. Let $R'=\frac{R}{\{0\}\times R_2}$. Then $J'=\frac{J}{\{0\}\times R_2}$
is a strongly 2-irreducible ideal of $R'$ by Theorem \ref{basic4}(5). Since $R'$ is ring-isomorphic to $R_1$ and $I_1\simeq J'$,
$I_1$ is a strongly 2-irreducible ideal of $R_1$. Suppose that $I_1=R_1$. Since $J$
is a proper ideal of $R$, $I_2\neq R_2$. By a similar argument as in the previous case, $I_2$ is
a strongly 2-irreducible ideal of $R_2$. Hence assume that $I_1\neq R_1$ and $I_2\neq R_2$. 
Suppose that $I_1$ is not a strongly irreducible ideal of $R_1$. Then there are $x,~y\in R_1$
such that $R_1x\cap R_1y\subseteq I_1$ but neither $x\in I_1$ nor $y\in I_1$. Notice that 
$(R_1x\times R_2)\cap(R_1\times\{0\})\cap(R_1y\times R_2)=(R_1x\cap R_1y)\times\{0\}\subseteq J$,
but neither $(R_1x\times R_2)\cap(R_1\times\{0\})=R_1x\times\{0\}\subseteq J$
nor $(R_1x\times R_2)\cap(R_1y\times R_2)=(R_1x\cap R_1y)\times R_2\subseteq J$
nor $(R_1\times\{0\})\cap(R_1y\times R_2)=R_1y\times\{0\}\subseteq J$,
which is a contradiction. Thus $I_1$ is a strongly irreducible ideal of $R_1$.
Suppose that $I_2$ is not a strongly irreducible ideal of $R_2$. Then there are $z,~w\in R_2$
such that $R_2z\cap R_2w\subseteq I_2$ but neither $z\in I_2$ nor $w\in I_2$. Notice that 
$(R_1\times R_2z)\cap(\{0\}\times R_2)\cap(R_1\times R_2w)=\{0\}\times(R_2z\cap R_2w)\subseteq J$,
but neither $(R_1\times R_2z)\cap(\{0\}\times R_2)=\{0\}\times R_2z\subseteq J$,
nor $(R_1\times R_2z)\cap(R_1\times R_2w)=R_1\times(R_2z\cap R_2w)\subseteq J$
nor $(\{0\}\times R_2)\cap(R_1\times R_2w)=\{0\}\times R_2w\subseteq J$,
which is a contradiction. Thus $I_2$ is a strongly irreducible ideal of $R_2$.\\
(2)$\Rightarrow$(1) If $J=I_1\times R_2$ for some strongly 2-irreducible ideal $I_1$ of $R_1$ or
$J=R_1\times I_2$ for some strongly 2-irreducible ideal $I_2$ of $R_2$, then it is clear that $J$ is
a strongly 2-irreducible ideal of $R$. Hence assume that $J=I_1\times I_2$ for some strongly irreducible
ideal $I_1$ of $R_1$ and some strongly irreducible ideal $I_2$ of $R_2$. Then 
$I'_1=I_1\times R_2$ and $I'_2=R_1\times I_2$ are strongly irreducible ideals of $R$. Hence 
$I'_1\cap I'_2=I_1\times I_2=J$ is a strongly 2-irreducible ideal of $R$ by Proposition \ref{basic}.
\end{proof}

\begin{theorem}
Let $R=R_1\times R_2\times\cdots\times R_n$, where $2\leq n<\infty$, and $R_1,R_2,
...,R_n$ are rings with $1\neq0$. Let $J$ be a proper ideal of $R$. Then the
following conditions are equivalent:
\begin{enumerate}
\item $J$ is a strongly 2-irreducible ideal of $R$.
\item Either $J=\times^n_{t=1}I_t$ such that for some $k\in\{1,2,...,n\}$, $I_k$ is a strongly 2-irreducible
ideal of $R_k$, and $I_t=R_t$ for every $t\in\{1,2,...,n\}\backslash\{k\}$ or $J=\times_{t=1}^nI_t$
such that for some $k,m\in\{1,2,...,n\}$, $I_k$ is a strongly irreducible ideal of $R_k$,
$I_m$ is a strongly irreducible ideal of $R_m$, and $I_t=R_t$ for every $t\in\{1,2,...,n\}\backslash\{k,m\}$.
\end{enumerate}
\end{theorem}
\begin{proof}
We use induction on $n$. Assume that $n=2$. Then the result is valid by
Theorem \ref{prod}. Thus let $3\leq n <\infty$ and assume that the result is valid when
$K=R_1\times\cdots\times R_{n-1}$. We prove the result when $R=K\times R_n$. By Theorem \ref{prod},
$J$ is a strongly 2-irreducible ideal of $R$ if and only if either $J=L\times R_n$ for some
strongly 2-irreducible ideal $L$ of $K$ or $J=K\times L_n$ for some strongly 2-irreducible
ideal $L_n$ of $R_n$ or $J=L\times L_n$ for some strongly irreducible ideal $L$ of $K$ and some strongly irreducible
ideal $L_n$ of $R_n$. Observe that a proper ideal $Q$ of $K$ is a strongly irreducible ideal of $K$ if
and only if $Q=\times_{t=1}^{n-1}I_t$ such that for some $k\in\{1,2,...,n-1\}$, $I_k$ is a strongly irreducible
ideal of $R_k$, and $I_t=R_t$ for every $t\in\{1,2,...,n-1\}\backslash\{k\}$. Thus the claim is now
verified. 
\end{proof}

\begin{lemma}\label{elem}
Let $R$ be a $GCD$-domain and $I$ be a proper ideal of $R$. The following conditions are equivalent:
\begin{enumerate}
\item $I$ is a singly strongly 2-irreducible ideal;
\item For every elements $x,y,z\in R$, $[x,y,z]\in I$ implies that $[x,y]\in I$ or $[x,z]\in I$ or $[y,z]\in I$.
\end{enumerate}
\end{lemma}
\begin{proof}
Since for every elements $x,~y$ of $R$ we have $Rx\cap Ry=R[x,y]$, there is nothing to prove.
\end{proof}

Now we study singly strongly 2-irreducible ideals of a $UFD$.
\begin{theorem}\label{basic3}
Let $R$ be a $UFD$, and let $I$ be a proper ideal of $R$. Then the following conditions hold:
\begin{enumerate}
\item $I$ is singly strongly 2-irreducible if and only if for each elements $x,y,z$ of $R$, $[x,y,z]\in I$
implies that either $[x,y]\in I$ or $[x,z]\in I$ or $[y,z]\in I$.
\item $I$ is singly strongly 2-irreducible if and only if $p_1^{n_1}p_2^{n_2}\cdots p_k^{n_k}\in I$, where 
$p_i$'s are distinct prime elements of $R$ and $n_i$'s are natural numbers, implies that $p_r^{n_r}p_s^{n_s}\in I$,
for some $1\leq r,s\leq k$.
\item If $I$ is a nonzero principal ideal, then $I$ is singly strongly 2-irreducible if and only if the 
generator of $I$ is a prime power or the product of two prime powers.
\item Every singly strongly 2-irreducible ideal is a 2-absorbing primary ideal.
\end{enumerate}
\end{theorem}
\begin{proof}
(1) By Lemma \ref{elem}.\\
(2) Suppose that $I$ is singly strongly 2-irreducible and $p_1^{n_1}p_2^{n_2}\cdots p_k^{n_k}\in I$ in which 
$p_i$'s are distinct prime elements of $R$ and $n_i$'s are natural numbers. Then 
$[p_1^{n_1},p_2^{n_2},\dots,p_k^{n_k}]=p_1^{n_1}p_2^{n_2}\cdots p_k^{n_k}\in I$.
Hence by part (1), there are $1\leq r,s\leq k$ such that $[p_r^{n_r},p_s^{n_s}]\in I$,
i.e., $p_r^{n_r}p_s^{n_s}\in I$.\\
For the converse, let $[x,y,z]\in I$ for some $x,y,z\in R\backslash\{0\}$. Assume that 
$x,~y$ and $z$ have prime decompositions as below,
\begin{eqnarray*}
x&=&p_1^{\alpha_1}p_2^{\alpha_2}\cdots p_k^{\alpha_k}q_1^{\beta_1}q_2^{\beta_2}\cdots q_s^{\beta_s},\\
y&=&p_1^{\gamma_1}p_2^{\gamma_2}\cdots p_k^{\gamma_k}r_1^{\delta_1}r_2^{\delta_2}\cdots r_u^{\delta_u},\\
z&=&p_1^{\varepsilon_1}p_2^{\varepsilon_2}\cdots p_{k'}^{\varepsilon_{k'}}
q_1^{\lambda_1}q_2^{\lambda_2}\cdots q_{s'}^{\lambda_{s'}}r_1^{\mu_1}r_2^{\mu_2}\cdots r_{u'}^{\mu_{u'}}
s_1^{\kappa_1}s_2^{\kappa_2}\cdots s_v^{\kappa_v},
\end{eqnarray*}
in which $0\leq k'\leq k$, $0\leq s'\leq s$ and $0\leq u'\leq u$. Therefore, 
$$\hspace{-2.2cm}[x,y,z]=p_1^{\nu_1}p_2^{\nu_2}\cdots p_{k'}^{\nu_{k'}}p_{k'+1}^{\omega_{k'+1}}\cdots p_k^{\omega_k}q_1^{\rho_1}q_2^{\rho_2}\cdots q_{s'}^{\rho_{s'}}$$
$$\hspace{2cm}q_{s'+1}^{\beta_{s'+1}}\cdots q_s^{\beta_s}r_1^{\sigma_1}r_2^{\sigma_2}\cdots r_{u'}^{\sigma_{u'}}r_{u'+1}^{\delta_{u'+1}}\cdots r_u^{\delta_u}
s_1^{\kappa_1}s_2^{\kappa_2}\cdots s_v^{\kappa_v}\in I,$$
where $\nu_i=max\{\alpha_i,\gamma_i,\varepsilon_i\}$ for every $1\leq i\leq k'$; $\omega_j=max\{\alpha_j,\gamma_j\}$ for every $k'< j\leq k$; $\rho_i=max\{\beta_i,\lambda_i\}$ for every $1\leq i\leq s'$;
$\sigma_i=max\{\delta_i,\mu_i\}$ for every $1\leq i\leq u'$. By part (2), we have twenty one cases. For example we investigate the following two cases. The other cases can be verified in a similar way.\\
{\bf Case 1.} For some $1\leq i,j\leq k'$, $p_i^{\nu_i}p_j^{\nu_j}\in I$. If $\nu_i=\alpha_i$ and $\nu_j=\alpha_j$, then clearly $x\in I$
and so $[x,y]\in I$. If $\nu_i=\alpha_i$ and $\nu_j=\gamma_j$, then $p_i^{\alpha_i}p_j^{\gamma_j}\mid[x,y]$
and thus $[x,y]\in I$. If $\nu_i=\alpha_i$ and $\nu_j=\varepsilon_j$, then $p_i^{\alpha_i}p_j^{\varepsilon_j}\mid[x,z]$
and thus $[x,z]\in I$.\\ 
{\bf Case 2.} Let $p_i^{\nu_i}p_j^{\omega_j}\in I$; for some $1\leq i\leq k'$ and $k'+1\leq j\leq k$. For $\nu_i=\alpha_i$, $\omega_j=\alpha_j$ we have $x\in I$ and so $[x,y]\in I$. For $\nu_i=\varepsilon_i$, $\omega_j=\gamma_j$ we have $[y,z]\in I$.\\
Consequently $I$ is singly strongly 2-irreducible, by part (1).\\
(3) Suppose that $I=Ra$ for some nonzero element $a\in R$. Assume that $I$ is singly strongly 2-irreducible. Let $a=p_1^{n_1}p_2^{n_2}\cdots p_k^{n_k}$ be a prime decomposition for $a$ such that $k>2$. By part (2) we have that $p_r^{n_r}p_s^{n_s}\in I$ for some $1\leq r,s\leq k$. Therefore $I=R(p_r^{n_r}p_s^{n_s})$.\\
 Conversely, if $a$ is a prime power, then $I$ is strongly
 irreducible ideal, by {\rm\cite[Theorem 2.2(3)]{Aziz}}. Hence $I$ is singly strongly 2-irreducible. Let $I=R(p^rq^s)$ for some prime elements $p,~q$ of $R$.
Assume that for some distinct prime elements $q_1,q_2,...,q_k$ of $R$ and natural numbers $m_1,m_2,...,m_k$,
$q_1^{m_1}q_2^{m_2}\cdots q_k^{m_k}\in I=R(p^rq^s)$. Then $p^rq^s\mid q_1^{m_1}q_2^{m_2}\cdots q_k^{m_k}$.
Hence there exists $1\leq i\leq k$ such that $p=q_i$ and $r\leq m_i$, also there exists $1\leq j\leq k$ 
such that $q=q_j$ and $s\leq m_j$. Then, since $p^rq^s\in I$, we have $q_i^{m_i}q_j^{m_j}\in I$. Now, by part (2),
$I$ is singly strongly 2-irreducible.\\
(4) Let $I$ be singly strongly 2-irreducible and $xyz\in I$ for some $x,y,z\in R\backslash\{0\}$.
Consider the following prime decompositions,
\begin{eqnarray*}
x&=&p_1^{\alpha_1}p_2^{\alpha_2}\cdots p_k^{\alpha_k}q_1^{\beta_1}q_2^{\beta_2}\cdots q_s^{\beta_s},\\
y&=&p_1^{\gamma_1}p_2^{\gamma_2}\cdots p_k^{\gamma_k}r_1^{\delta_1}r_2^{\delta_2}\cdots r_u^{\delta_u},\\
z&=&p_1^{\varepsilon_1}p_2^{\varepsilon_2}\cdots p_{k'}^{\varepsilon_{k'}}
q_1^{\lambda_1}q_2^{\lambda_2}\cdots q_{s'}^{\lambda_{s'}}r_1^{\mu_1}r_2^{\mu_2}\cdots r_{u'}^{\mu_{u'}}
s_1^{\kappa_1}s_2^{\kappa_2}\cdots s_v^{\kappa_v},
\end{eqnarray*}
in which $0\leq k'\leq k$, $0\leq s'\leq s$ and $0\leq u'\leq u$. By these representations we have,

$xyz=p_1^{\alpha_1+\gamma_1+\varepsilon_1}p_2^{\alpha_2+\gamma_2+\varepsilon_2}\cdots p_{k'}^{\alpha_{k'}+\gamma_{k'}+\varepsilon_{k'}}p_{k'+1}^{\alpha_{k'+1}+\gamma_{k'+1}}
$\vspace*{2mm}\\
\hspace*{2cm}$\cdots p_{k}^{\alpha_{k}+\gamma_{k}}q_1^{\beta_1+\lambda_1}q_2^{\beta_2+\lambda_2}\cdots q_{s'}^{\beta_{s'}+\lambda_{s'}}q_{s'+1}^{\beta_{s'+1}}\cdots q_s^{\beta_s}$
\vspace*{2mm}\\
\hspace*{3cm}$r_1^{\delta_1+\mu_1}r_2^{\delta_2+\mu_2}\cdots r_{u'}^{\delta_{u'}+\mu_{u'}}
r_{u'+1}^{\delta_{u'+1}}\cdots r_{u}^{\delta_{u}}s_1^{\kappa_1}s_2^{\kappa_2}\cdots s_v^{\kappa_v}\in I.$\\
Now, apply part (2). We investigate some cases that can be happened, the other cases similarly lead us to the claim that 
$I$ is 2-absorbing primary. 
First, assume for some $1\leq i,j\leq k'$, $p_i^{\alpha_i+\gamma_i+\varepsilon_i}
p_j^{\alpha_j+\gamma_j+\varepsilon_j}\in I$. Choose a natural number $n$ such that $n\geq max\{\frac{\alpha_i+\gamma_i}{\varepsilon_i},\frac{\alpha_j+\gamma_j}{\varepsilon_j}\}$. With this choice we have 
$(n+1)\varepsilon_i\geq\alpha_i+\gamma_i+\varepsilon_i$ and $(n+1)\varepsilon_j\geq\alpha_j+\gamma_j+\varepsilon_j$,
so $p_i^{(n+1)\varepsilon_i}p_j^{(n+1)\varepsilon_j}\in I$. Then $z^{n+1}\in I$, so $z\in\sqrt{I}$.
The other one case; assume that for some $1\leq i\leq k'$ and $k'+1\leq j\leq k$, $p_i^{\alpha_i+\gamma_i+\varepsilon_i}
p_j^{\alpha_j+\gamma_j}\in I$. Choose a natural number $n$ such that $n\geq max\{\frac{\alpha_i+\varepsilon_i}{\gamma_i},\frac{\alpha_j}{\gamma_j}\}$. With this choice we have $(n+1)\gamma_i\geq\alpha_i+\gamma_i+\varepsilon_i$ and
$(n+1)\gamma_j\geq\alpha_j+\gamma_j$,
thus $p_i^{(n+1)\gamma_i}p_j^{(n+1)\gamma_j}\in I$. Then $y^{n+1}\in I$, so $y\in\sqrt{I}$.
Assume that $p_i^{\alpha_i+\gamma_i}s_j^{\kappa_j}\in I$, for some $k'+1\leq i\leq k$ and some $1\leq j\leq v$.
Let $n$ be a natural number where $n\geq\frac{\gamma_i}{\alpha_i}$, then $(n+1)\alpha_i\geq\alpha_i+\gamma_i$.
Hence $p_i^{(n+1)\alpha_i}s_j^{(n+1)\kappa_j}\in I$ which shows that $xz\in\sqrt{I}$. Suppose that for some $s'+1\leq i\leq s$ and $u'+1\leq j\leq u$, $q_i^{\beta_i}r_j^{\delta_j}\in I$.
Then, clearly $xy\in I$.
\end{proof}

\begin{corollary}\label{basic12}
Let $R$ be a $UFD$.
\begin{enumerate}
\item Every principal ideal of $R$ is a singly strongly 2-irreducible ideal  if and only if
it is a 2-absorbing primary ideal.
\item Every singly strongly 2-irreducible ideal of $R$ can be generated by a set of 
elements of the forms $p^n$ and $p_i^{n_i}p_j^{n_j}$ in which $p,p_i,p_j$ are some prime elements of $R$
and $n,n_i,n_j$ are some natural numbers.
\item Every 2-absorbing ideal of $R$ is a singly strongly 2-irreducible ideal.
\end{enumerate}
\end{corollary}
\begin{proof}
(1) Suppose that $I$ is singly strongly 2-irreducible ideal. By Theorem \ref{basic3}(4), 
$I$ is a 2-absorbing primary ideal. Conversely, let $I$ be a nonzero 2-absorbing primary ideal. 
Let $I=Ra$, where $0\neq a\in I$. Assume that $a=p_1^{n_1}p_{2}^{n_2}\cdots p_{k}^{n_k}$
be a prime decomposition for $a$. If $k>2$, then since $p_1^{n_1}p_{2}^{n_2}\cdots p_{k}^{n_k}\in I$
and $I$ is a 2-absorbing primary ideal, there exist a natural number $n$, and integers $1\leq i,j\leq k$
such that $p_i^{nn_i}p_j^{nn_j}\in I$, say $i=1$ and $j=2$. Therefore $p_3\mid p_1^{nn_1}p_2^{nn_2}$ which is a contradiction. Therefore $k=1$ or 2, that is 
$I=Rp_1^{n_1}$ or $I=R(p_1^{n_1}p_2^{n_2})$, respectively. Hence by Theorem \ref{basic3}(3),
$I$ is singly strongly 2-irreducible.\\
(2) Let $X$ be a generator set for a singly strongly 2-irreducible ideal of $I$, and let $x$
be a nonzero element of $X$. Assume that $x=p_1^{n_1}p_2^{n_2}\cdots p_k^{n_k}$
be a prime decomposition for $x$ such that $k\geq2$. By Theorem \ref{basic3}(2), for some $1\leq i,j\leq k$, 
we have $p_i^{n_i}p_{j}^{n_j}\in I$, and then $Rx\subseteq Rp_i^{n_i}p_{j}^{n_j}\subseteq I$.
Consequently, $I$ can be generated by a set of elements of the forms $p^n$ and  $p_i^{n_i}p_{j}^{n_j}$.\\
(3) is a direct consequence of Theorem \ref{basic3}(2).
\end{proof}

The following example shows that  in part (1) of Corollary \ref{basic12} the condition that $I$ is principal is necessary.
Moreover, the converse of part (2) of this corollary need not be true. 
\begin{example}
Let $F$ be a field and $R=F[x,y,z]$, where $x,~y$ and $z$ are independent indeterminates. 
We know that $R$ is a $UFD$. Suppose that $I=\langle x,y^2,z^2\rangle$. Since $\sqrt{\langle x,y^2,z^2\rangle}=
\langle x,y,z\rangle$ is a maximal ideal of $R$, $I$ is a primary ideal and so is a 2-absorbing primary ideal.
Notice that $(x+y+z)yz\in I$, but neither $(x+y+z)y\in I$ nor $(x+y+z)z\in I$ nor $yz\in I$.
Consequently, $I$ is not singly strongly 2-irreducible, by Theorem \ref{basic3}(2).
\end{example}

\section*{\bf Acknowledgments}
The authors are grateful to the referee of this paper for his/her careful reading and
comments.


\begin{thebibliography}{20}

\bibitem{AA} 
D. D. Anderson and D. F. Anderson, 
\newblock Generalized $GCD$-domains, 
\newblock {\em Comment. Math. Univ. St. Pauli} \textbf{28} (1979) 215--221.

\bibitem{AB1} 
D. F. Anderson and A. Badawi, 
\newblock On $n$-absorbing ideals of commutative rings, 
\newblock {\em Comm. Algebra} \textbf{39} (2011) 1646--1672.

\bibitem{Aziz} 
A. Azizi, 
\newblock Strongly irreducible ideals, 
\newblock {\em J. Aust. Math. Soc.} \textbf{84} (2008) 145--154.

\bibitem{B} 
A. Badawi, 
\newblock On $2$-absorbing ideals of commutative rings, 
\newblock {\em Bull. Austral. Math. Soc.} \textbf{75} (2007) 417--429.

\bibitem{Bt} 
A. Badawi, U. Tekir and E. Yetkin, 
\newblock On $2$-absorbing primary ideals in commutative rings, 
\newblock {\em Bull. Korean Math. Soc.} \textbf{51}
(2014), no. 4, 1163--1173.

\bibitem{YB} 
A. Badawi and A. Yousefian Darani, 
\newblock On weakly $2$-absorbing ideals of commutative rings, 
\newblock{\it Houston J. Math.} \textbf{39} (2013) 441--452.

\bibitem{f} 
L. Fuchs and L. Salce, 
\newblock {\em Modules over non-noetherian domains}, 
\newblock Mathematical Surveys and Monographs, Vol. 84, 2000.

\bibitem{hei} 
W. J. Heinzer, L. J. Ratliff Jr and D. E. Rush, 
\newblock Strongly irreducible ideals of a commutative ring,
\newblock {\em J. Pure Appl. Algebra} \textbf{166} (2002) 267--275.

\bibitem{H} 
T. W. Hungerford, 
\newblock {\em Algebra}, 
\newblock Springer-Verlag, 1974.


\bibitem{jen} 
C. Jensen, 
\newblock Arithmetical rings, 
\newblock {\em Acta Math. Acad. Sci. Hungar.} \textbf{17} (1966) 115--123.

\bibitem{jeo} 
Y. C. Jeon, N. K. Kim and Y. Lee, 
\newblock On fully idempotent rings,
\newblock{\em Bull. Korean Math. Soc.} \textbf{47} (2010), no. 4, 715--726. 

\bibitem{lam} 
T. Y. Lam, 
\newblock {\em Lectures on modules and rings}, 
\newblock Graduate Texts in Mathematics, Springer-Verlag, 1998.

\bibitem{L} 
D. Lorenzini, 
\newblock {\em An invitation to arithmetic geometry,} 
\newblock Graduate Studies in Mathematics, Vol. 9, 1996.

\bibitem{M} 
H. Mostafanasab, E. Yetkin, U. Tekir and A. Yousefian Darani, 
\newblock  On 2-absorbing primary submodules of modules over commutative rings,
\newblock {\em An. Sti. U. Ovid. Co-mat.} To appear.


\bibitem{Sh} 
R.Y. Sharp, 
\newblock {\em Steps in commutative algebra}, 
\newblock Second edition, Cambridge University Press, Cambridge, 2000. 

\bibitem{YFM1} 
A. Yousefian Darani, and H. Mostafanasab,  
\newblock Co-$2$-absorbing preradicals and submodules, 
\newblock {\it J. Algebra Appl.} To appear.

\bibitem{YFM2} 
A. Yousefian Darani, and H. Mostafanasab,  
\newblock On $2$-absorbing preradicals, 
\newblock {\it J. Algebra Appl.} \textbf{14} (2015), no. 2, 1550017 (22 pages.) 

\bibitem{YFP} 
A. Yousefian Darani, and E. R. Puczy{\l}owski,  
\newblock On 2-absorbing commutative semigroups and their applications to rings,
\newblock {\it Semigroup Forum} \textbf{86} (2013) 83--91.

\bibitem{YF} 
A. Yousefian Darani and F. Soheilnia, 
\newblock $2$-absorbing and weakly $2$-absorbing submoduels, 
\newblock {\it Thai J. Math.} \textbf{9} (2011), no. 3, 577--584.

\bibitem{YF2} 
A. Yousefian Darani and F. Soheilnia, 
\newblock On $n$-absorbing submodules, 
\newblock{\em Math. Commun}. \textbf{17} (2012) 547--557.

\end{thebibliography}
\end{document}